\documentclass[11pt,a4paper]{amsart}

\usepackage[english]{babel} 
\usepackage{amsmath, amssymb} 
\usepackage[all,cmtip]{xy} 
\usepackage{hyperref}

\DeclareMathOperator{\Hom}{Hom}

\DeclareMathOperator{\LMod}{LMod}
\DeclareMathOperator{\Perf}{Perf}
\DeclareMathOperator{\Tor}{Tor}
\DeclareMathOperator{\Map}{Map}
\DeclareMathOperator{\Ho}{Ho}
\DeclareMathOperator{\Ind}{Ind}

\DeclareMathOperator{\Idem}{Idem}
\DeclareMathOperator{\Spec}{Spec}

\begin{document}

\newcommand{\cat}[1]{\mathrm{#1}}
\newcommand{\sSet}{\cat{sSet}}

\newcommand{\Q}{\mathbf{Q}}
\newcommand{\C}{\mathbf{C}}
\newcommand{\N}{\mathbf{N}}
\newcommand{\Z}{\mathbf{Z}}
\newcommand{\R}{\mathbf{R}}

\newcommand{\cO}{\mathcal{O}}
\newcommand{\cA}{\mathcal{A}}
\newcommand{\cB}{\mathcal{B}}
\newcommand{\cG}{\mathcal{G}}
\newcommand{\cT}{\mathcal{T}}
\newcommand{\cF}{\mathcal{F}}
\newcommand{\cX}{\mathcal{X}}

\newcommand{\op}{\mathrm{op}}
\newcommand{\Cat}{\mathrm{Cat}}
\newcommand{\exact}{\mathrm{ex}}
\newcommand{\st}{\mathrm{st}}

\newcommand{\id}{\mathrm{id}}

\newcommand{\pr}{\mathrm{pr}}

\newcommand{\Cone}{\mathrm{Cone}}

\newcommand{\laxtimes}[1]{\mathop{\times\mkern-13mu\raise1.3ex\hbox{$\scriptscriptstyle\to$}_{#1}}}

\renewcommand{\labelenumi}{(\roman{enumi})}

\newtheorem{thm}{Theorem}
\newtheorem*{thm*}{Theorem}
\newtheorem{cor}[thm]{Corollary}
\newtheorem*{cor*}{Corollary}
\newtheorem{lemma}[thm]{Lemma}
\newtheorem{prop}[thm]{Proposition}

\theoremstyle{definition}
\newtheorem{dfn}[thm]{Definition}
\newtheorem*{dfn*}{Definition}

\theoremstyle{remark}

\newtheorem{claim}[thm]{Claim}

\newtheorem{rem}[thm]{Remark}
\newtheorem*{rem*}{Remark}
\newtheorem{rems}[thm]{Remarks}
\newtheorem*{ex*}{Example}
\newtheorem{ex}[thm]{Example}
\newtheorem{recollection}[thm]{Recollection}


\title{Excision in algebraic \textit{K}-theory revisited}

\author{Georg Tamme}
\email{georg.tamme@ur.de}
\address{Fakult\"at f\"ur Mathematik, Universit\"at Regensburg, 93040 Regensburg, Germany}

\thanks{The author is supported by the CRC 1085 \emph{Higher Invariants} (Universit\"at Regensburg) funded by the DFG}

 \begin{abstract}
By a theorem of Suslin, a Tor-unital (not necessarily unital) ring satisfies excision in algebraic $K$-theory. 
We give a new and direct proof of Suslin's result based on an exact sequence of categories of perfect modules.
In fact, we prove a more general descent result for a pullback square of ring spectra and any localizing invariant. 
Besides Suslin's result, this also contains Nisnevich descent of algebraic $K$-theory for affine schemes as a special case.
Moreover, the role of the Tor-unitality condition becomes very transparent.
   \end{abstract}

\maketitle

\section*{Introduction}

One of the main achievements in the algebraic $K$-theory of rings is the solution of the excision problem, first rationally by Suslin--Wodzicki \cite{SuslinWod} and later integrally by Suslin \cite{Suslin}: For a two-sided ideal  $I$ in a unital ring $A$ one defines the relative  $K$-theory spectrum $K(A,I)$ as the homotopy fibre of the map of $K$-theory spectra $K(A) \to K(A/I)$, so that its homotopy groups $K_{*}(A,I)$ fit in a long exact sequence
\[
\dots \to K_{i}(A,I) \to K_{i}(A) \to K_{i}(A/I) \to K_{i-1}(A,I) \to \dots
\]
If $I$ is a not necessarily unital ring, one defines $K_{*}(I) := K_{*}(\Z \ltimes I, I)$ where $\Z \ltimes I$ is the unitalization of $I$. 
For every unital ring $A$ containing $I$ as a two-sided ideal there is a  canonical map  $\Z \ltimes I \to A$. It induces a  map $K_{*}(I) \to K_{*}(A,I)$ and one says that $I$ satisfies excision in algebraic $K$-theory if this map is an isomorphism for all such $A$. 

Equivalently, $I$ satisfies excision in algebraic $K$-theory if, for every ring $A$ containing $I$ as a two-sided ideal and any ring homomorphism $A \to B$ sending $I$ isomorphically onto an ideal of $B$, the pullback square of rings 
\begin{equation} \label{diag:Milnor-square}
\begin{split}
\xymatrix@C-0.3cm@R-0.3cm{
A \ar[r] \ar[d]  & A'   \ar[d]   \\
B \ar[r] &  B'
}
\end{split}
\end{equation}
where $A' = A/I$, $B' = B/I$
induces a homotopy cartesian square of non-connective $K$-theory spectra
\begin{equation} \label{diag:K-spectra-square}
\begin{split}
\xymatrix@C-0.3cm@R-0.3cm{
K(A) \ar[r] \ar[d]  & K(A')   \ar[d]   \\
K(B) \ar[r] &  K(B').
}
\end{split}
\end{equation}

A ring $I$ is called \emph{Tor-unital} if $\Tor^{\Z \ltimes I}_{i}(\Z,\Z) = 0$ for all $i > 0$. 
Every unital ring is Tor-unital, since if $I$ is unital, then $\Z \ltimes I \cong \Z \times I$ and the projection to $\Z$ is flat. 

\begin{thm}[Suslin]\label{thm:thm1}
If $I$ is Tor-unital, then $I$ satisfies excision in algebraic $K$-theory.
\end{thm}
In fact,  both  statements are equivalent \cite[Thm.~A]{Suslin}.
For $\Q$-algebras, this was proven before by Suslin and Wodzicki \cite[Thm.~A]{SuslinWod}. Wodzicki \cite{Wodzicki}  gives many  examples of  Tor-unital $\Q$-algebras, for instance all $C^{*}$-algebras. These results are the main ingredients in the proof of Karoubi's conjecture about algebraic and topological $K$-theory of stable $C^{*}$-algebras in \cite[Thm.~10.9]{SuslinWod}.
On the other hand, by work of Morrow \cite{Morrow} ideals $I$  in commutative noetherian rings are pro-Tor-unital in the sense that the pro-groups $\{ \Tor_{i}^{\Z \ltimes I^{n}}(\Z,\Z)\}_{n}$ vanish for all $i > 0$.

Suslin's proof of Theorem~\ref{thm:thm1}
uses the description of algebraic $K$-theory in terms of Quillen's plus-construction
and relies on a careful study of the homology of affine groups.
By completely different methods we prove the following generalization of Theorem~\ref{thm:thm1}: 
\begin{thm}\label{thm:thm2}
Assume that \eqref{diag:Milnor-square} is a homotopy pullback square of  ring spectra such that the multiplication map  $A' \otimes_{A} A' \to A'$ is an equivalence. Then the square \eqref{diag:K-spectra-square} of non-connective $K$-theory spectra is homotopy cartesian.
 \end{thm}
Here the tensor denotes the derived tensor product, and $K$-theory is the non-connective $K$-theory of perfect modules.

\begin{ex}
	\label{ex:Suslin-Nisnevich}
Assume that  \eqref{diag:Milnor-square} is a diagram of discrete rings. When viewed as a diagram of ring spectra,
 this  is a homotopy pullback square if and only if the induced sequence of abelian groups
\[
0 \to A \to A' \oplus B \to B' \to 0
\]
is exact. The  multiplication map $A' \otimes_{A} A' \to A'$ is an equivalence if and only if  $\Tor^{A}_{i}(A', A') = 0$ for all $i > 0$ and the ordinary tensor product of $A'$ with itself over $A$ is isomorphic to $A'$ via the multiplication.

There are two basic cases where both conditions are satisfied: The first one is  that $A' = A/I$ for a Tor-unital two-sided ideal $I$ in $A$ (see Example~\ref{ex:classical-Milnor}). This gives Suslin's result. 
The second one is that \eqref{diag:Milnor-square} is an elementary affine Nisnevich square, i.e.,
 all rings are commutative, $A' = A[f^{-1}]$ is a localization of $A$,  $A \to B$ is an \'etale map inducing an isomorphism $A/(f) \cong B/(f)$, and $B' = B[f^{-1}]$ (see Example~\ref{ex:Nisnevich}). 
Note that by \cite[Prop.~2.3.2]{AHW} the family of coverings of the form $\{\Spec(A[f^{-1}]) \to \Spec(A), \Spec(B) \to \Spec(A) \}$
generate the Nisnevich topology on the category of affine schemes (of finite presentation over some base).
 Thus Theorem~\ref{thm:thm2} also implies Nisnevich descent for the algebraic $K$-theory of affine schemes.
\end{ex}

In general, the condition that $A' \otimes_{A} A' \to A'$ be an equivalence is equivalent to $\LMod(A) \to \LMod(A')$  being a localization, where $\LMod$ denotes the $\infty$-category of left modules in spectra.
In particular, under this condition $\LMod(A')$ is a Verdier quotient of $\LMod(A)$.
The usual method that is used, for example,  to produce localization sequences in $K$-theory (see \cite[\S3]{Schlichting} for an overview, \cite[Thm.~0.5]{Neeman-Ranicki} for the case of a non-commutative localization where a similar condition on Tor-groups appears), 
would be to apply Neeman's generalization of Thomason's localization  theorem \cite[Thm.~2.1]{Neeman-connection} in order to deduce that  also 
the induced functor on the subcategories of compact objects, which are precisely the perfect modules, $\Perf(A) \to \Perf(A')$ is a Verdier quotient.
 However, Neeman's theorem does not apply here, since the kernel of $\LMod(A) \to \LMod(A')$ need not be compactly generated. Indeed, there is an example by Keller \cite[\S2]{Keller} of a ring map $A \to A'$ satisfying the hypotheses of Theorem~\ref{thm:thm2}, where this kernel has no non-zero compact objects at all and $\Perf(A')$ is not a Verdier quotient of $\Perf(A)$. 
 
Instead, under the conditions of Theorem~\ref{thm:thm2} we  prove a derived version of Milnor patching (Theorem~\ref{thm:pullback-modules-tor-unital-case}) saying that \eqref{diag:Milnor-square} induces a pullback diagram of $\infty$-categories of left modules, i.e.,
\[
\LMod(A) \simeq \LMod(A') \times_{\LMod(B')} \LMod(B).
\]
Its proof is inspired by a similar patching result for connective modules over connective ring spectra due to Lurie \cite[Thm.~16.2.0.2]{sag}.
We use this to show that $\LMod(A)$ can be embedded as a full subcategory  in the lax pullback $\LMod(A') \laxtimes{\LMod(B')} \LMod(B)$ (see Section~\ref{section1}) and to identify the Verdier quotient with $\LMod(B')$.
Now  the Thomason--Neeman theorem applies  and gives
  an exact sequence of small stable $\infty$-categories 
\begin{equation}
	\label{eq:exact-sequence-perfect-modules}
\Perf(A) \xrightarrow{i} \Perf(A') \laxtimes{\Perf(B')} \Perf(B)  \xrightarrow{\pi} \Perf(B'),
\end{equation}
i.e., the composite $\pi\circ i$ is zero and the induced functor from the Verdier quotient of the middle term by $\Perf(A)$ to
$\Perf(B')$ is an equivalence up to idempotent completion. 
This
 implies the assertion of Theorem~\ref{thm:thm2} not only for algebraic $K$-theory, but for any  invariant  which can be defined for small stable $\infty$-categories and which sends exact sequences of such to fibre sequences. 
In fact, in Section~\ref{section1}  we prove the existence of the analog of the exact sequence~\eqref{eq:exact-sequence-perfect-modules} for any so-called excisive square of small stable $\infty$-categories (Theorem~\ref{thm:main-theorem-new}). In Section~\ref{section2} we then prove that any square of ring spectra satisfying the hypotheses of Theorem~\ref{thm:thm2} yields an excisive square of $\infty$-categories of perfect modules (see Theorem~\ref{thm:excisive-square-ring-spectra}). These are the two main results of the paper.

\begin{rem}
The failure of excision in $K$-theory is measured in (topological) cyclic homology: Corti{\~n}as \cite{Cortinas-Obstruction} proved
that the fibre of the rational Goodwillie--Jones Chern character from rational algebraic $K$-theory to negative cyclic homology satisfies excision, i.e., sends the pullback square of rings \eqref{diag:Milnor-square}  with $B \to B'$ surjective to a homotopy pullback square of spectra without any further condition. Geisser and Hesselholt \cite{GeisserHesselholt} proved the analogous result with finite coefficients, replacing the Goodwillie--Jones Chern character by the cyclotomic trace map from $K$-theory to topological cyclic homology.
Both use  pro versions of the results of Suslin and Wodzicki.
Building on these results, Dundas and Kittang \cite{Dundas-Kittang-1, Dundas-Kittang-2} prove that the fibre of the cyclotomic trace satisfies excision also for connective ring spectra, and with integral coefficients (under the technical assumption that both, $\pi_{0}(B) \to \pi_{0}(B')$ and $\pi_{0}(A') \to \pi_{0}(B')$ are surjective).

In this general situation, i.e., without assuming any Tor-unitality condition, one still has the sequence \eqref{eq:exact-sequence-perfect-modules}, but the induced functor $f$ from the Verdier quotient to $\Perf(B')$ need not be an equivalence up to idempotent completion.
It would therefore be interesting to find conditions on an invariant $E$ that guarantee that $E(f)$ is still an equivalence.\footnote{We give a sufficient condition in a forthcoming article with Markus Land:  it suffices that the natural map $E(C) \to E(\pi_{0}(C))$ is an equivalence for any connective $E_{1}$-ring spectrum $C$.}
From the results mentioned above we know that $E(f)$ is an equivalence for $E$ the fibre of the cyclotomic trace.

\end{rem}
  
 We use $\infty$-categorical language. More concretely, we use the model of quasi-categories, which are the fibrant objects for the Joyal model structure on simplicial sets, as developed by Joyal \cite{Joyal} and Lurie in his books \cite{htt,halg,sag}.
  
  \smallskip
  
\noindent\textit{Acknowledgements.}
I  would like to express my sincere gratitude to the referee for the efforts taken to improve both the exposition and the results of this paper.
The referee gave a hint which led to a simplification of the proof of the main result of the first version of this paper, and also suggested to formulate the general categorical Theorem~\ref{thm:excisive-square} in terms of excisive squares and to deduce the excision result via Theorem~\ref{thm:excisive-square-ring-spectra}.
I would also like to thank Justin Noel and Daniel Sch{\"a}ppi for discussions about (lax) pullbacks of $\infty$-categories.%

\section{Pullbacks and exact sequences of stable $\infty$-categories}
\label{section1}

In this section, we discuss the pullback and the lax pullback of a diagram $A \to C \leftarrow B$ of $\infty$-categories. In the stable case, we relate these by exact sequences.
We further prove our first main result (Theorem~\ref{thm:excisive-square}) saying that any excisive square of small stable $\infty$-categories (see Definition~\ref{dfn:excisive-square}) yields a pullback square upon applying any localizing invariant.

Let $I = \Delta[1] \in \sSet$ be the standard simplicial 1-simplex. For any $\infty$-category $C$, we denote by $C^{I} = \mathrm{Fun}(I,C)$ the arrow category of $C$. The inclusion $\{0,1\} \subseteq I$ induces the source and target maps $s,t\colon C^{I} \to C$.

Consider a diagram of $\infty$-categories
\begin{equation}\label{diag:diag1}\begin{split}
\xymatrix{
& B \ar[d]^{q}\\
A \ar[r]^{p} & C.
}
\end{split}
\end{equation}
\begin{dfn}
The \emph{lax pullback}  $A \laxtimes{C} B$ of \eqref{diag:diag1} is defined 
via the pullback diagram
\begin{equation}\label{diag:diag-lax-pullback} 
\begin{split}
\xymatrix{
A \laxtimes{C} B \ar[d]_{(\pr_{1},\pr_{2})} \ar[r]^{\pr_{3}} & C^{I} \ar[d]^{(s,t)} \\
A\times B \ar[r]^{p \times q} & C\times C
}
\end{split}
\end{equation}
in simplicial sets.
\end{dfn}
By \cite[Ch.~5, Thm.~A]{Joyal} the map  $C^{I} \xrightarrow{(s,t)} C \times C$ is a categorical fibration, i.e., a fibration in the Joyal model structure. Since the lower and upper right corners in \eqref{diag:diag-lax-pullback} are $\infty$-categories, this implies that $A \laxtimes{C} B$ is indeed an $\infty$-category, and that \eqref{diag:diag-lax-pullback} is homotopy cartesian with respect to the Joyal model structure.

\begin{rem}\label{rem:mapping-spaces}
The objects of $A \laxtimes{C} B$ are triples of the form $(a,b, g \colon p(a) \to q(b))$ where $a$, $b$ are objects of $A$, $B$ respectively and $g$ is a morphism $p(a) \to q(b)$ in $C$.
If $(a,b,g)$ and $(a',b',g')$ are two objects of $A \laxtimes{C} B$, the mapping space between these sits in a homotopy cartesian diagram of spaces
\[
\xymatrix{
\Map( (a,b,g), (a',b',g') ) \ar[r] \ar[d]   &   \Map_{C^{I}}(g,g') \ar[d]   \\
\Map_{A}(a,a') \times \Map_{B}(b,b')   \ar[r]    &    \Map_{C}( p(a), p(a') ) \times \Map_{C}( q(b), q(b') ).
}
\] 
Indeed, using Lurie's $\Hom^{\mathrm{R}}$-model for the mapping spaces \cite[\S 1.2.2]{htt} gives a cartesian diagram of simplicial sets in which the right vertical map is a Kan fibration by \cite[Lemma 2.4.4.1]{htt}.
\end{rem}

\begin{rem}
Denote by $C^{(I)} \subseteq C^{I}$ the full subcategory spanned by the equivalences in $C$. It follows from \cite[Prop.~5.17]{Joyal} that the pullback of the diagram
\[
\xymatrix{
  & C^{(I)} \ar[d]^{(s,t)}  \\
  A \times B \ar[r]^{p\times q}  & C \times C
}
\]
in simplicial sets models the homotopy pullback of $\infty$-categories $A \times_{C} B$. In particular, we can identify $A \times_{C} B$ with the full subcategory of $A \laxtimes{C} B$ spanned by those objects $(a,b,g)$ where $g$ is an equivalence in $C$.
\end{rem}

\begin{lemma}\phantomsection
	\label{lem:stability-of-pullbacks}
\begin{enumerate}

\item Let $K$ be a simplicial set and $\delta\colon K \to A \laxtimes{C} B$  a diagram. If the compositions of $\delta$ with the projections to $A$ and $B$ admit colimits and these colimits are preserved by $p$ and $q$ respectively, then $\delta$ admits a colimit, which is preserved by the projections to $A$ and $B$.
The same statement holds for diagrams in $A \times_{C} B$.

\item If $A$ and $B$ are idempotent complete, then $A \laxtimes{C} B$ and $A \times_{C} B$ are idempotent complete.

\item If $A$, $B$, and $C$ are presentable and $p$ and $q$ commute with colimits, then both $\infty$-categories $A \laxtimes{C} B$ and $A \times_{C} B$ are presentable. Moreover, a functor from a presentable $\infty$-category $D$ to $A \times_{C} B$ or $A \laxtimes{C} B$ preserves colimits if and only if the  compositions with the projections to $A$ and $B$ do. 

\item If $A$, $B$, and $C$ are stable, and $p$ and $q$ are exact, then both $\infty$-categories $A \laxtimes{C} B$ and $A \times_{C} B$ are stable.
\end{enumerate}
\end{lemma}

For the definition of a presentable $\infty$-category see \cite[Def.~5.5.0.1]{htt}, for that of an idempotent complete $\infty$-category  \cite[\S4.4.5]{htt}, and for that of a stable $\infty$-category \cite[Def.~1.1.1.9]{halg}.

\begin{proof}
(i) The assumptions and \cite[Prop.~5.1.2.2]{htt} (applied to the projection $C \times I \to I$) imply that the composition of $\delta$ with the projection to $C^{I}$ also admits a colimit. Now the claim follows from \cite[Lemmas~5.4.5.4, 5.4.5.2]{htt}.

(ii) Let $\Idem$ be the nerve of the 1-category with a single object $X$ and $\Hom(X,X) = \{ \id_{X}, e\}$, where $e \circ e = e$. An $\infty$-category $D$ is idempotent complete if and only if any diagram $\Idem \to D$ admits a colimit.\footnote{This is Cor.~4.4.5.15 in the 2017 version of HTT, available at the author's homepage.} It follows from \cite[Prop.~4.4.5.12, Lemma~4.3.2.13]{htt} that every functor between $\infty$-categories $D \to D'$ preserves colimits of diagrams indexed by $\Idem$. Hence the claim follows from part (i).

By construction of the lax pullback, it suffices to check the remaining assertions for  pullbacks  and  functor categories.

(iii) For the functor category see \cite[Prop.~5.5.3.6, Cor.~5.1.2.3]{htt} and  for the pullback \cite[Prop.~5.5.3.12]{htt}.

(iv) See \cite[Prop.~1.1.3.1]{halg} for the functor category, \cite[Prop.\ 1.1.4.2]{halg} for the pullback.
\end{proof}

From now on, we will mainly be concerned with stable $\infty$-categories. Recall that by \cite[Thm.~1.1.2.14]{halg} the homotopy category $\Ho(A)$ of a stable $\infty$-category $A$ is a triangulated category.

\begin{recollection} \label{recollection}
We recall the $\infty$-categorical version of Verdier quotients. For a detailed discussion see \cite[\S 5]{Blumberg-Tabuada-Gepner}.
Let  $\Pr^{\mathrm{L}}_{\st}$ denote the $\infty$-category of presentable stable $\infty$-categories and left adjoint (equivalently, colimit preserving) functors, and let $\Cat^{\exact}_{\infty}$  be the $\infty$-category of small stable $\infty$-categories and exact functors.
Both admit small colimits. 
Given a fully faithful functor $A \to B$ in either of these, $B/A$ denotes its cofibre. By \cite[Prop.~5.9, 5.14]{Blumberg-Tabuada-Gepner} the functor $B \to B/A$ induces an equivalence of the Verdier quotient $\Ho(B) / \Ho(A)$  with $\Ho(B/A)$.

A sequence $A \to B \to C$ in $\Pr^{\mathrm{L}}_{\st}$ or $\Cat^{\exact}_{\infty}$ is called \emph{exact} if the composite is zero, $A \to B$ is fully faithful, and the induced map $B/A \to C$ is an equivalence after idempotent completion.
It follows from \cite[Prop.~5.10]{Blumberg-Tabuada-Gepner} and the above that $A \to B \to C$ is exact if and only $\Ho(A) \to \Ho(B) \to \Ho(C)$ is exact (up to factors) in the sense of triangulated categories (see e.g.~\cite[Def.~3.1.10]{Schlichting}).

If $C$ is a localization of $B$, i.e., the functor $B \to C$ has a fully faithful right adjoint, and $A \to B$ induces an equivalence of $A$ with the kernel of $B \to C$, i.e., the full subcategory of objects of $B$ that map to a zero object in $C$, then $A \to B \to C$ is exact.
\end{recollection}

For the remainder of this section, we
assume that \eqref{diag:diag1} is a diagram of stable $\infty$-categories and exact functors. 

The pair of functors $B \to A \times B$, $b \mapsto (0,b)$, and $B \to C^{I}$, $b \mapsto ( 0 \to q(b) )$, induces a functor $r\colon B \to A \laxtimes{C} B$. 
Similarly, the functors $A \to A \times B$, $a \mapsto (a,0)$, and $A \to C^{I}$, $a \mapsto ( p(a) \to 0 )$, induce a functor $s \colon A \to A \laxtimes{C} B$.

\begin{prop}\label{prop:split-exact-sequence}
Assume that \eqref{diag:diag1} is a diagram of stable $\infty$-categories and exact functors. We have  a split exact sequence
\[
\xymatrix@1{
B \ar[r]_-{r} & A \laxtimes{C} B \ar[r]_-{\pr_{1}}  \ar@/_{1pc}/[l]_-{\pr_{2}} & A ,  \ar@/_{1pc}/[l]_-{s}
}
\]
i.e., the sequence is exact, $\pr_{2}$ and $s$ are right adjoints of $r$, $\pr_{1}$, respectively, and $\id_{B} \simeq \pr_{2}\circ r $, $\pr_{1} \circ s \simeq \id_{A}$ via unit and counit, respectively.
\end{prop}
\begin{proof} 
By construction, we have $\id_{B} = \pr_{2} \circ r$ and we claim that this is a unit transformation for the desired adjunction (see \cite[Prop.~5.2.2.8]{htt}). That is, we have to show that for any object $b$ in $B$ and $(a',b',g') $ in $A \laxtimes{C} B$
the  map $\Map( r(b), (a', b', g') ) \to \Map(b, b') $ induced by $\pr_{2}$ is an equivalence. This map is the second component of the left vertical map in the diagram
\[
\xymatrix{
\Map( r(b) , (a',b',g') ) \ar[r] \ar[d]   &   \Map( (0 \to q(b) ),g') \ar[d]   \\
\Map(0,a') \times \Map(b,b')   \ar[r]    &    \Map( 0 , p(a') ) \times \Map( q(b), q(b') ).
}
\]
By Remark \ref{rem:mapping-spaces} this diagram is homotopy cartesian. 
Since 
the functor $C \to C^{I}$, $c \mapsto (0 \to c)$, is a left adjoint of $t\colon C^{I} \to C$, the right vertical map is an equivalence. Hence the left vertical map is an equivalence (use  that $\Map(0,a')$ and $\Map(0,p(a'))$ are contractible).

Similarly, one shows that $s$ is a right adjoint of $\pr_{1}$. Since the counit $\pr_{1}\circ s \to \id_{A}$ is an equivalence, $s$ is fully faithful. Since moreover $r$ induces an equivalence of $B$ with the kernel of $\pr_{1}$, the sequence in the statement of the lemma is exact by Recollection~\ref{recollection}.
\end{proof}

We let $\pi$ be the composition of functors $A \laxtimes{C} B \xrightarrow{\pr_{3} } C^{I} \xrightarrow{\Cone} C$, where 
 $\Cone\colon C^{I} \to C$ sends a morphism in $C$ to its cofibre.

\begin{prop}\label{prop:exact-sequence-of-presentable-cats}
Assume that \eqref{diag:diag1} is a diagram of stable $\infty$-categories and exact functors. 
Assume furthermore that  $q\colon B \to C$ admits a fully faithful right adjoint $v \colon C \to B$.
Then the composite
\[
\rho\colon C \xrightarrow{v} B \xrightarrow{r} A \laxtimes{C} B
\]
is a fully faithful right adjoint of $\pi$.
\end{prop}
\begin{proof}
Since $v$ is fully faithful by assumption, and $r$ is fully faithful by Proposition~\ref{prop:split-exact-sequence}, the functor $\rho$ is fully faithful.
The functor $\Cone\colon C^{I} \to C$ has the right adjoint $\beta$ mapping $c$ to $(0 \to c)$ \cite[Rem.~1.1.1.8]{halg}.
By the construction of $r$ we have  a canonical equivalence $\pr_{3} \circ r \simeq \beta \circ q$. Hence the counit of the adjoint pair $(\Cone,\beta)$ induces a natural transformation $\pi \circ r = \Cone \circ \pr_{3} \circ r \simeq \Cone \circ \beta \circ q \to q$ and hence $\pi \circ r \circ v \to q \circ v$. Composing with the counit of the adjoint pair $(q,v)$ we get a natural transformation $\eta\colon \pi \circ \rho = \pi \circ r \circ v \to \id_{C}$. We claim that $\eta$ is a counit transformation for the desired adjunction. This will imply the claim by \cite[Prop.~5.2.2.8]{htt}. We thus have to show that the composition
\begin{equation}\label{eq:counit-composition}
\Map( (a,b,g), \rho(c) ) \xrightarrow{\pi} \Map( \pi((a,b,g)), \pi(\rho(c)) ) \xrightarrow{\eta} \Map(\pi((a,b,g)), c)
\end{equation}
is an equivalence for every object $(a,b,g)$ in $A \laxtimes{C} B$ and any object $c$ in $C$.

From Remark~\ref{rem:mapping-spaces} we have a homotopy pullback square of spaces
\begin{equation}\label{diag:dddd}\begin{split}
\xymatrix{
\Map( (a,b,g) , \rho(c) ) \ar[r]^{\pr_{3}} \ar[d]  & \Map( g, (0 \to q(v(c))) ) \ar[d]  \\
\Map( a,0) \times \Map(b, v(c) ) \ar[r]  & \Map(p(a),0 ) \times \Map( q(b), q(v(c)) ).
}
\end{split}
\end{equation}
Since $v$ is fully faithful, $q(v(c)) \simeq c$ and the lower horizontal map is an equivalence by adjunction.
Hence the upper horizontal map $\pr_{3}$ is an equivalence, too. 
The $(\Cone,\beta)$-adjunction yields an equivalence
\begin{equation}\label{eq:Cone-beta-adjunction}
\Map(g, (0 \to q(v(c))) ) \xrightarrow{\simeq} \Map( \Cone(g), q(v(c))).
\end{equation}
By construction, \eqref{eq:counit-composition} is the composition of the equivalences $\pr_{3}$ in \eqref{diag:dddd} and \eqref{eq:Cone-beta-adjunction} and the map induced by the counit $q(v(c)) \to c$,  which is an equivalence by fully faithfulness of $v$. Hence \eqref{eq:counit-composition} is an equivalence, as desired.
\end{proof}

\begin{cor} 
	\label{cor:exact-sequence-of-presentable-cats}
Assume that \eqref{diag:diag1} is a diagram in $\Pr^{\mathrm{L}}_{\st}$. 
If the right adjoint of $B \to C$ is fully faithful, then
 the sequence 
\[
A \times_{C} B \to A \laxtimes{C} B  \xrightarrow{\pi} C
\]
is exact.
\end{cor}
\begin{proof}
An object $(a,b,g)$ of $A \laxtimes{C} B$ belongs to $A \times_{C} B$ if and only if $g$ is an equivalence, if and only if $\Cone(g) \simeq 0$. This shows that the composite is trivial and that $A \times_{C} B$ is precisely the kernel of $\pi$. 
The claim now follows, since  $\pi$ admits a fully faithful right adjoint  by Proposition~\ref{prop:exact-sequence-of-presentable-cats}.
\end{proof}

Let $A'$ be a small stable $\infty$-category. Then the $\infty$-category $\Ind(A')$ of Ind-objects  of $A'$ \cite[Def.~5.3.5.1]{htt} is presentable \cite[Thm.~5.5.1.1]{htt} and  stable \cite[Prop.~1.1.3.6]{halg}.
A stable $\infty$-category $A$ is called \emph{compactly generated} if there exists a small stable $\infty$-category $A'$ and an equivalence $\Ind(A') \simeq A$ (see \cite[Def.~5.5.7.1]{htt} and the text following it). If this is the case, then $A' \to A$ induces an equivalence of the idempotent completion of $A'$ \cite[\S 5.1.4]{htt} with the full stable subcategory $A^{\omega}$ of the compact objects in $A$ \cite[Lemma~5.4.2.4]{htt}.
In particular, if $A$ is compactly generated, $A^{\omega}$ is (essentially) small and $\Ind(A^{\omega}) \simeq A$.
Whether a stable $\infty$-category is idempotent complete or compactly generated only depends  on its homotopy category \cite[Lemma~1.2.4.6, Rem.~1.4.4.3]{halg}.

\begin{prop}
	\label{prop:lax-pullback-cptly-generated}
Assume that \eqref{diag:diag1} is a diagram in $\Pr^{\mathrm{L}}_{\st}$ in which $A$ and $B$ are compactly generated and the functors $p\colon A \to C$ and $q\colon B \to C$ map compact objects to compact objects. Then $A \laxtimes{C} B$ is compactly generated as well and  $(A \laxtimes{C} B)^{\omega} \simeq  A^{\omega} \laxtimes{C^{\omega}} B^{\omega}$.
\end{prop}

\begin{proof}
By Lemma~\ref{lem:stability-of-pullbacks}(iii) the $\infty$-category $A \laxtimes{C} B$ is presentable and hence admits all small colimits.
Let $D' := A^{\omega} \laxtimes{C^{\omega}} B^{\omega}$. This is an (essentially) small full stable subcategory of $A \laxtimes{C} B$. It follows from \cite[Lemmas 5.4.5.7, 5.3.4.9]{htt} that $D'$ consists of compact objects in $A \laxtimes{C} B$. Hence the induced functor $\Ind(D') \to A \laxtimes{C} B$ is fully faithful.
Since the functors $r\colon B \to A \laxtimes{C} B$ and $s\colon A \to A \laxtimes{C} B$ preserve colimits by Lemma~\ref{lem:stability-of-pullbacks}(iii) and since $A$ and $B$ are compactly generated, it follows that the essential image of $\Ind(D')$ in $A \laxtimes{C} B$ contains $A$ and $B$. Proposition~\ref{prop:split-exact-sequence} implies that every object $X$ of $A \laxtimes{C} B$ sits in a fibre sequence $X' \to X \to X''$ with $X' \in B$ and $X'' \in A$.
Hence the essential image of $\Ind(D')$ must be all of $A \laxtimes{C} B$, and hence the latter is compactly generated.
Since $A^{\omega}$ and $B^{\omega}$ are idempotent complete, so is $D'$ by Lemma~\ref{lem:stability-of-pullbacks}(ii). 
Hence $D' \simeq (A \laxtimes{C} B)^{\omega}$.
\end{proof}

\begin{dfn}
	\label{dfn:excisive-square}
An \emph{excisive square} of small stable $\infty$-categories is a commutative square 
\begin{equation}
	\label{diag:excisive-square}
	\begin{split}
	\xymatrix@C-0.2cm@R-0.2cm{
	D \ar[r] \ar[d] & B \ar[d]^-{q}\\
	A \ar[r]^-{p} & C
	}
	\end{split}
\end{equation}
in $\Cat^{\exact}_{\infty}$ such that the induced square 
\begin{equation}
	\label{diag:ind-excisive-square}
	\begin{split}
	\xymatrix@C-0.3cm@R-0.3cm{
	\Ind(D) \ar[r] \ar[d] & \Ind(B) \ar[d]^{}\\
	\Ind(A) \ar[r]^{} & \Ind(C)
	}
	\end{split}
\end{equation}
in $\Pr^{\mathrm{L}}_{\st}$ is a pullback square and $\Ind(B) \to \Ind(C)$ is a localization, i.e., its right adjoint is fully faithful.
\end{dfn}

The following is the categorical version of our first main result.
\begin{thm}
	\label{thm:main-theorem-new}
Assume that \eqref{diag:excisive-square} is an excisive square of small stable $\infty$-categories. Then there is an exact sequence
\begin{equation}
	\label{eq:fundamental-exact-sequence-new}
D \xrightarrow{i} A \laxtimes{C} B \xrightarrow{\pi} C.
\end{equation}
\end{thm}
\begin{proof}
If we apply Corollary~\ref{cor:exact-sequence-of-presentable-cats} to the pullback diagram \eqref{diag:ind-excisive-square}, we get the exact sequence
\[
\Ind(D) \to \Ind(A) \laxtimes{\Ind(C)} \Ind(B) \to \Ind(C)
\]
in $\Pr^{\mathrm{L}}_{\st}$. Clearly, the first and the third term in this sequence are compactly generated. Proposition~\ref{prop:lax-pullback-cptly-generated} implies that also the middle term is compactly generated, and that the functors preserve compact objects.
Recall from Recollection~\ref{recollection} that we can test exactness on the level of  homotopy categories.
Thus we may apply the Thomason--Neeman localization theorem \cite[Thm.~2.1]{Neeman-connection} to conclude that the induced sequence of compact objects is exact.
But up to idempotent completion this is exactly \eqref{eq:fundamental-exact-sequence-new}.
\end{proof}

We now apply this to localizing invariants.
\begin{dfn}
	\label{dfn:weakly-localizing}
A \emph{weakly localizing invariant} is a functor 
\[
E \colon \Cat^{\exact}_{\infty} \to T
\]
from $\Cat^{\exact}_{\infty}$ to some stable $\infty$-category $T$ which sends exact sequences in $\Cat^{\exact}_{\infty}$ to fibre sequences in $T$.
\end{dfn}

\begin{ex}
Any localizing invariant in the sense of \cite{Blumberg-Tabuada-Gepner} is weakly localizing. 
Concrete examples are non-connective algebraic $K$-theory \`a la Bass-Thomason \cite[\S9.1]{Blumberg-Tabuada-Gepner}, topological Hochschild homology $THH$ \cite[\S10.1]{Blumberg-Tabuada-Gepner}, or $p$-typical topological cyclic homology $TC$ for some prime $p$ \cite[\S10.3]{Blumberg-Tabuada-Gepner}, \cite{Blumberg-Mandell}. In all these examples $T$ is the $\infty$-category of spectra.
\end{ex}

\begin{thm}\label{thm:excisive-square}
Assume that \eqref{diag:excisive-square} is an excisive square of small stable $\infty$-categories, and let 
$E\colon \Cat^{\exact}_{\infty} \to T$ be a weakly localizing invariant.
 Then the induced square in $T$
\begin{equation}
	\label{diag:thm-excisive-square}
	\begin{split}
\xymatrix@C-0.3cm@R-0.3cm{
E(D) \ar[r] \ar[d] &  E(B) \ar[d] \\
E(A) \ar[r] & E(C)
}
\end{split}
\end{equation}
is cartesian.
\end{thm}
\begin{proof}
Applying $E$ to the exact sequence \eqref{eq:fundamental-exact-sequence-new} provided by Theorem~\ref{thm:main-theorem-new} yields the fibre sequence
\begin{equation}
	\label{eq:E-excisive-square}
E(D) \xrightarrow{E(i)} E( A \laxtimes{C} B) \xrightarrow{E(\pi)} E(C)
\end{equation}
in $T$. On the other hand, applying $E$ to the split exact sequence of Proposition~\ref{prop:split-exact-sequence} gives an equivalence
\begin{equation}
	\label{eq:E-split-exact-sequence}
E(s) \oplus E(r) \colon E(A) \oplus E(B) \xrightarrow{\simeq} E( A \laxtimes{C} B )
\end{equation}
with inverse induced by the projections $\pr_{1}$, $\pr_{2}$. Combining \eqref{eq:E-excisive-square} and \eqref{eq:E-split-exact-sequence}, we get a fibre sequence
\begin{equation}
	\label{eq:E-fibre-sequence}
E(D) \to E(A) \oplus E(B) \to E(C)
\end{equation}
where the first map is induced by the given functors $D \to A$ and $D \to B$. The map $E(A) \to E(C)$ is induced by the functor $a \mapsto \Cone( p(a) \to 0 ) \simeq \Sigma p(a)$.
Since the endofunctor $\Sigma\colon C \to C$ induces $-\id$ on $E(C)$, the map $E(A) \to E(C)$ in \eqref{eq:E-fibre-sequence} is the negative of the map induced by the functor $p\colon A \to C$. Finally, the map $E(B) \to E(C)$ in \eqref{eq:E-fibre-sequence} is induced by the functor $b \mapsto \Cone( 0 \to q(b) ) \simeq q(b)$. Thus \eqref{eq:E-fibre-sequence} being a fibre sequence in $T$ implies that \eqref{diag:thm-excisive-square} is cartesian.
\end{proof}

\begin{rem}
This theorem can also be used to prove the Mayer-Vietoris property of algebraic $K$-theory for the Zariski topology \cite[Thm.~8.1]{thomason} for quasi-compact quasi-separated schemes without using Thomason's localization theorem \cite[Thm.~7.4]{thomason}. Together with Example~\ref{ex:Suslin-Nisnevich} one may then deduce Nisnevich descent for  noetherian schemes in general.
\end{rem}

\section{Application to ring spectra}
\label{section2}

In this section, we apply the constructions of Section~\ref{section1} to the $\infty$-categories of (perfect) modules over an $E_{1}$-ring spectrum, discuss Tor-unitality, and  
we prove our second main result (Theorem~\ref{thm:excisive-square-ring-spectra}) saying that a pullback square of ring spectra where one map is Tor-unital (Definition~\ref{dfn:tor-unital}) yields an excisive square upon applying $\Perf(-)$. From this we finally deduce Theorems~\ref{thm:thm1} and~\ref{thm:thm2} of the Introduction.

The $\infty$-categories of $E_{1}$-ring spectra and their modules are discussed in \cite[Ch.~7]{halg}.
For an $E_{1}$-ring spectrum $A$, we write $\LMod(A)$ for the stable $\infty$-category of left $A$-module spectra, which we will simply call left $A$-modules henceforth.
A left $A$-module is called \emph{perfect} if it belongs to the smallest stable subcategory $\Perf(A)$ of $\LMod(A)$ which contains $A$ and is closed under retracts. By \cite[Prop.~7.2.4.2]{halg}, $\LMod(A)$ is compactly generated and the compact objects are precisely the perfect $A$-modules.

\begin{ex}
Any discrete ring $A$ can be considered as an $E_{1}$-ring spectrum. Then $\Ho(\LMod(A))$ is equivalent to the unbounded derived category of $A$ in the classical sense \cite[Rem.~7.1.1.16]{halg}.
\end{ex}

\begin{dfn}
	\label{dfn:tor-unital}
A map $f\colon A \to A'$  of  $E_{1}$-ring spectra  is called \emph{Tor-unital} if  the following equivalent conditions are satisfied:
\begin{enumerate}
\item The map $A'\otimes_{A} A' \to A'$ given by  multiplication is an equivalence.
\item The map $A' \to A' \otimes_{A} A'$ induced from $A \to A'$ by $A' \otimes_{A} (-)$ is an equivalence.
\item If $I$ is  the fibre of $A \to A'$ in $\LMod(A)$, we have $A' \otimes_{A} I \simeq 0$.
\end{enumerate}
\end{dfn}

We have the following easy but important further characterization of Tor-unitality:
\begin{lemma}\label{lem:Tor-unital-implies-ff}
A morphism $A \to A'$ of $E_{1}$-ring spectra is Tor-unital if and only if the forgetful functor $\LMod(A') \to \LMod(A)$ is fully faithful.
\end{lemma}
\begin{proof}
By \cite[Prop.~4.6.2.17]{halg} the forgetful functor $v$ is right adjoint to $A' \otimes_{A} - \colon \LMod(A) \to \LMod(A')$. It is fully faithful if and only if the counit $A' \otimes_{A} M \to M$ is an equivalence for every $A'$-module $M$. Taking $M=A'$, we see that fully faithfulness of $v$ implies Tor-unitality of $A \to A'$. The converse follows, since $\LMod(A')$ is generated by $A'$ under small colimits and finite limits, and the tensor product preserves both.
\end{proof}

Now consider any pullback square of  $E_{1}$-ring spectra
\begin{equation}\label{diag:ring-spectra}
\begin{split}
\xymatrix@C-0.3cm@R-0.3cm{
A \ar[r] \ar[d] & A' \ar[d] \\
B \ar[r] & B'.
}
\end{split}
\end{equation}

\begin{lemma} \label{lem:stability-of-Tor-unitality}
Assume that \eqref{diag:ring-spectra} is a pullback square of $E_{1}$-ring spectra in which $A \to A'$ is Tor-unital.
Then also $B \to B'$ is Tor-unital. Moreover, the canonical map $A' \otimes_{A} B \to A' \otimes_{A} B'$ 
 induced from $B \to B'$ is an equivalence.
\end{lemma}

See Remark~\ref{rem:Luries-patching} for a partial converse.

\begin{proof}
Write $I$ for the fibre of $A \to A'$. 
Since $A \to A'$ is Tor-unital, $A' \otimes_{A} I \simeq 0$. As by assumption \eqref{diag:ring-spectra} is a pullback square,  the fibre of $B \to B'$ is equivalent (as left $A$-module) to $I$, hence $A' \otimes_{A} B \to A' \otimes_{A} B'$ is an equivalence, too. 
By Lemma~\ref{lem:Tor-unital-implies-ff} the counit  $A' \otimes_{A} M \to M$ is an equivalence for every $A'$-module $M$. In particular, $A' \otimes_{A} B' \to B'$ is an equivalence. 
Summing up, the canonical map $A' \otimes_{A} B \to B'$ is an equivalence. Thus
\[
B' \otimes_{B} B' \simeq (A' \otimes_{A} B) \otimes_{B} B' \simeq A' \otimes_{A} B' \simeq B'
\]
and $B \to B'$ is Tor-unital.
\end{proof}

\begin{ex}
	\label{ex:Milnor-square-as-pullback-of-ring-spectra}
	\label{ex:classical-Milnor}
Let $A \to B$ be a morphism of discrete unital rings sending a two-sided ideal $I$ of $A$ isomorphically onto an ideal  of $B$. Then the Milnor square
\[
\xymatrix@C-0.3cm@R-0.3cm{
A \ar[r] \ar[d]  & A/I   \ar[d]   \\
B \ar[r] &  B/I
}
\]
is a pullback diagram in rings.
Since $B \to B/I$ is surjective, this diagram is also a pullback  when considered as a diagram of $E_{1}$-ring spectra.
The map $A \to A/I$ is Tor-unital if and only if $\Tor_{i}^{A}( A/I, A/I ) = 0$ for all $i > 0$.

In particular, if the discrete, not necessarily unital ring $I$ is Tor-unital in the classical sense that $\Tor^{\Z\ltimes I}_{i}(\Z,\Z) = 0$ for all $i > 0$, then Lemma~\ref{lem:stability-of-Tor-unitality} applied to the Milnor square
\[
\xymatrix@C-0.3cm@R-0.3cm{
\Z \ltimes I \ar[r] \ar[d]   & \Z  \ar[d]   \\
A \ar[r]   & A/I
}
\]
implies that $A \to A/I$ is Tor-unital for any ring $A$ containing $I$ as a two-sided ideal.
\end{ex}

\begin{ex}
	\label{ex:Nisnevich}
Assume that $A$ is a commutative, unital discrete ring, and let $f \in A$. Then $A \to A[f^{-1}]$ is Tor-unital.
Assume further  that $A \to B$ is an \'etale ring map which induces an isomorphism $A/(f) \xrightarrow{\sim} B/(f)$. Then 
the diagram 
\[
\xymatrix@C-0.3cm@R-0.3cm{
A  \ar[r] \ar[d]   & A[f^{-1}]  \ar[d]   \\
B \ar[r]   & B[f^{-1}],
}
\]
viewed as a diagram of $E_{1}$-ring spectra, is a pullback square.
Indeed, this is equivalent to the exactness of the sequence
\[
0 \to A \to A[f^{-1}] \oplus B \to B[f^{-1}] \to 0,
\]
which may be checked directly. Alternatively, one may use the Mayer--Vietoris exact sequence of \'etale cohomology groups
\[
0 \to A \to A[f^{-1}] \oplus B \to B[f^{-1}] \to H^{1}_{\mathrm{\acute{e}t}}(\Spec(A), \mathcal{O}_{\Spec(A)}),
\]
which may be deduced from \cite[Prop.~III.1.27]{Milne}, together with the vanishing of the higher \'etale cohomology of quasi-coherent sheaves on affine schemes.
\end{ex}

The following is a derived version of Milnor patching:
\begin{thm}
	\label{thm:pullback-modules-tor-unital-case}
Assume that \eqref{diag:ring-spectra} is a pullback square of  $E_{1}$-ring spectra where the morphism $A \to A'$  is  Tor-unital.
Then extension of scalars induces an equivalence
\begin{equation*}
\LMod(A) \simeq \LMod(A') \times_{\LMod(B')} \LMod(B).
\end{equation*}
\end{thm}
\begin{proof}
Let $F$ be the functor $\LMod(A) \to \LMod(A') \times_{\LMod(B')} \LMod(B)$ induced by extension of scalars.
Since both $\infty$-categories are presentable and $F$ preserves colimits by Lemma~\ref{lem:stability-of-pullbacks}(iii), $F$ admits a right adjoint $G$. Explicitly, if $(M,N,g)$ is an object of $ \LMod(A') \times_{\LMod(B')} \LMod(B)$, then $G(M,N,g)$ is the pullback in left $A$-modules
\[
G(M,N,g) \simeq M \times_{B'\otimes_{B}N} N
\]
where the map $M \to B' \otimes_{B} N$ is the  composition $M \to B' \otimes_{A'} M \xrightarrow{g} B' \otimes_{B} N$.
We claim that the unit 
\[
P \to (A' \otimes_{A} P) \times_{B' \otimes_{B} ( B \otimes_{A} P )} (B \otimes_{A} P)
\]
of the adjunction is an equivalence for any $A$-module $P$. Since also  $G$  commutes with colimits, it suffices to check this for $P = A$. In that case the claim follows from the assumption that \eqref{diag:ring-spectra} is a pullback square.
Hence $F$ is fully faithful.

It now suffices to show that the right adjoint $G$ of $F$ is conservative. 
For this it is enough to show that $G$ detects zero objects. So let $(M,N,g)$ be an object of the pullback and assume that $G(M,N,g) \simeq 0$.
There is a 
fibre sequence of left $A$-modules
\[
G(M,N,g) \to M \oplus N \to B' \otimes_{B} N
\]
and hence the map
\begin{equation}
		\label{eq:MoplusN}
M \oplus N \xrightarrow{\simeq} B' \otimes_{B} N
\end{equation}
 is an equivalence. Extending scalars from $A$ to $A'$ we get an equivalence
\begin{equation}
	\label{eq:A'otimesMoplusN}
A' \otimes_{A} M \oplus A' \otimes_{A} N \xrightarrow{\simeq} A' \otimes_{A} B' \otimes_{B} N.
\end{equation}
From  Lemma \ref{lem:stability-of-Tor-unitality} we know that $A' \otimes_{A} B \to A' \otimes_{A} B'$ is an equivalence.  
Since $\LMod(B)$ is generated by $B$ under colimits and finite limits, we conclude that $A' \otimes_{A} P \to A' \otimes_{A} B' \otimes_{B} P$ is an equivalence for every left $B$-module $P$. Applying this with $P=N$, we see that the restriction of \eqref{eq:A'otimesMoplusN} to the second summand is an equivalence.
Hence $A' \otimes_{A} M \simeq 0$. Since $M$ is an $A'$-module, Lemma \ref{lem:Tor-unital-implies-ff} implies that the counit is an equivalence $A' \otimes_{A} M \simeq  M$, i.e., $M \simeq 0$. But then also $B' \otimes_{B} N \simeq B' \otimes_{A'} M \simeq 0$, and hence $N \simeq 0$ by \eqref{eq:MoplusN}.
\end{proof}

\begin{rem}
	\label{rem:Luries-patching}
Without the Tor-unitality assumption Theorem~\ref{thm:pullback-modules-tor-unital-case} does not hold, see \cite[Warning 16.2.0.3]{sag} for a counter example. 

However, if one assumes instead that 
\eqref{diag:ring-spectra} is a pullback square of connective ring spectra with $\pi_{0}(B) \to \pi_{0}(B')$ surjective, then 
\cite[Prop.\ 16.2.2.1]{sag} implies that restricting the functors $F$ and $G$ from the proof of Theorem~\ref{thm:pullback-modules-tor-unital-case} to the subcategories of connective modules  gives inverse equivalences
\begin{equation*}
\LMod(A)_{\geq 0}  \leftrightarrows \LMod(A')_{\geq 0} \times_{\LMod(B')_{\geq 0}} \LMod(B)_{\geq 0}.
\end{equation*}
One can use this to show that in this situation, Tor-unitality of $B \to B'$ implies Tor-unitality of $A \to A'$: 
Let $I$ be the fibre of $B \to B'$. Since $\pi_{0}(B) \to \pi_{0}(B')$ is surjective, $I$ is connective. Since $B \to B'$ is Tor-unital,  $B' \otimes_{B} I \simeq 0$.
Hence we may view $(0, I, 0)$ as an object of the pullback $\LMod(A')_{\geq 0} \times_{\LMod(B')_{\geq 0}} \LMod(B)_{\geq 0}$.
The functor $G$ sends $(0,I,0)$ to the $A$-module $0 \times_{0} I \simeq I$. By the above the counit $F(I) \simeq F(G(0,I,0))\to (0,I,0)$ is an equivalence. Looking at the first component we deduce that $A' \otimes_{A} I \to 0$ is an equivalence, i.e., $A \to A'$ is Tor-unital.
\end{rem}

\begin{thm}
	\label{thm:excisive-square-ring-spectra}
Assume that \eqref{diag:ring-spectra} is a pullback square of  $E_{1}$-ring spectra where the morphism $A \to A'$  is  Tor-unital. Then the square 
\begin{equation}
	\label{diag:excisive-square-perf}
	\begin{split}
	\xymatrix@C-0.3cm@R-0.3cm{
	\Perf(A) \ar[r] \ar[d] & \Perf(B) \ar[d]  \\
	\Perf(A') \ar[r] & \Perf(B')
	}
	\end{split}
\end{equation}
is excisive. In particular, if $E \colon \Cat^{\exact}_{\infty} \to T$ is a weakly localizing invariant, then the induced square
\begin{equation*}
	\label{diag:E-excisive-square-perf}
	\begin{split}
	\xymatrix@C-0.3cm@R-0.3cm{
	E(\Perf(A)) \ar[r] \ar[d] & E(\Perf(B)) \ar[d]  \\
	E(\Perf(A')) \ar[r] & E(\Perf(B'))
	}
	\end{split}
\end{equation*}
in $T$ is cartesian.
\end{thm}
\begin{proof}
Applying $\Ind$ to  diagram \eqref{diag:excisive-square-perf} yields the diagram
\begin{equation*}
	\label{diag:excisive-square-mod}
	\begin{split}
	\xymatrix@C-0.3cm@R-0.3cm{
	\LMod(A) \ar[r] \ar[d] & \LMod(B) \ar[d]  \\
	\LMod(A') \ar[r] & \LMod(B')
	}
	\end{split}
\end{equation*}
This is a pullback diagram by Theorem~\ref{thm:pullback-modules-tor-unital-case}.
As $A \to A'$ is Tor-unital, so is $B \to B'$ by Lemma~\ref{lem:stability-of-Tor-unitality}. 
Hence the right adjoint of $\LMod(B) \to \LMod(B')$, which is the forgetful functor, is fully faithful by Lemma~\ref{lem:Tor-unital-implies-ff}. So the square~\eqref{diag:excisive-square-perf} is excisive. Now the second assertion follows by applying Theorem~\ref{thm:excisive-square}.
\end{proof}

\begin{proof}[Proof of Theorems~\ref{thm:thm1} and \ref{thm:thm2}]
If we apply Theorem~\ref{thm:excisive-square-ring-spectra} with $E=K$, we immediately get Theorem~\ref{thm:thm2}.

Now let $I$ be a ring which is Tor-unital in the classical sense, and let $A$ be any unital ring containing $I$ as a two-sided ideal.
Then the Milnor square
\[
\xymatrix@C-0.3cm@R-0.3cm{
\Z \ltimes I \ar[r] \ar[d] & \Z \ar[d]  \\
A \ar[r] & A/I,
}
\]
viewed as square of $E_{1}$-ring spectra, is a pullback square (see Example~\ref{ex:Milnor-square-as-pullback-of-ring-spectra}).
By assumption, the top horizontal map is Tor-unital in our sense. 
Hence we may apply Theorem~\ref{thm:thm2} to deduce that the map on relative $K$-groups $K_{*}(I) = K_{*}(\Z \ltimes I, I) \to K_{*}(A, I)$ is an isomorphism.
\end{proof}

\providecommand{\bysame}{\leavevmode\hbox to3em{\hrulefill}\thinspace}

\end{document}